\documentclass[10pt]{amsart}

\usepackage{amsthm}
\usepackage{amssymb}
\usepackage{amsmath}
\usepackage{mathtools}
\usepackage{pdfpages}
\usepackage{float}
\usepackage[colorlinks=true,unicode=true,linktocpage,bookmarksopen,hypertexnames=false]{hyperref}
\usepackage{caption}
\usepackage[capitalise]{cleveref}
\usepackage{tikz-cd}

\DeclareMathOperator{\Ker}{Ker}

\newcommand{\Aut}{\operatorname{Aut}}

\makeatletter
\numberwithin{equation}{section}
\numberwithin{figure}{section}
\numberwithin{table}{section}

\newtheorem{theorem}{Theorem}[section]
\newtheorem*{theorem*}{Theorem}
\newtheorem{lemma}[theorem]{Lemma}
\newtheorem{corollary}[theorem]{Corollary}
\newtheorem{proposition}[theorem]{Proposition}

\newtheorem{definition}[theorem]{Definition}

\newtheorem{remark}[theorem]{Remark}
\newtheorem{example}[theorem]{Example}
\makeatother

\title{Hopf formulae for homology of skew braces}

\author{Marino Gran}
\author{Thomas Letourmy}
\author{Leandro Vendramin}

\address[Gran]{Institut de Recherche en Mathématique et Physique , Université Catholique de Louvain, Chemin du Cyclotron 2, B-1348 Louvain-la-Neuve, Belgium}
\email{marino.gran@uclouvain.be}

\address[Letourmy]{Département de Mathéematique, Université Libre de Bruxelles, Boulevard du Triomphe, B-1050 Brussels, Belgium; and 
Department of Mathematics and Data Science, Vrije Universiteit Brussel, Pleinlaan 2, 1050 Brussel, Belgium}
\email{thomas.letourmy@ulb.be}

\address[Vendramin]{Department of Mathematics and Data Science, Vrije Universiteit Brussel, Pleinlaan 2, 1050 Brussel, Belgium}
\email{Leandro.Vendramin@vub.be}

\subjclass[2010]{Primary: 18E13, 18G50, 18E50; 20J05}
\keywords{Skew brace, Radical ring, Non-abelian homology, Hopf formulae, Central extension, Commutator}

\begin{document}

\begin{abstract}
The variety of skew braces contains several interesting subcategories as subvarieties, as for instance the varieties of radical rings, of groups and of abelian groups. In this article the methods of non-abelian homological algebra are applied to establish some new Hopf formulae for homology of skew braces, where the coefficient functors are the reflectors from the variety of skew braces to each of the three above-mentioned subvarieties. The corresponding central extensions of skew braces are characterized in purely algebraic terms, leading to some new results, such as an explicit Stallings--Stammbach exact sequence associated with any exact sequence of skew braces, and a new result concerning central series.
\end{abstract}

\maketitle 

\section{Introduction}
Skew braces appeared originally in connection with the study of set-theoretic solutions to the Yang--Baxter equation \cite{MR3647970,Rump}. Now their applications go far beyond this domain,
as they appear in several different areas; see for example \cite{MR3291816}. 

A \emph{skew brace} \cite{MR3647970} is a triple $(A,+,\circ)$, 
where $(A,+)$ and $(A,\circ)$ are groups such that
the compatibility condition 
$a\circ (b+c)=a\circ b-a+a\circ c$ holds for all $a,b,c\in A$.

Skew braces form a variety of universal algebras $\mathsf{SKB}$, and generalise at the same time groups and radical rings. Concretely, given a group $(G,\cdot)$ one can give $G$ a skew brace structure by taking $+=\cdot$ and $\circ=\cdot$. In particular, this means that the variety $\mathsf{Grp}$ of groups is a subvariety of the variety $\mathsf{SKB}$ of skew braces. On the other hand, a radical ring $(R,+,\cdot)$ is a (not necessarily commutative) ring without unit such that the operation $a\circ b=a+a\cdot b+b$ is a group operation. In this case, the triple $(R,+,\circ)$ is a skew brace. We will recall in Section \ref{sect:Radrings} a characterisation of radical rings \cite{Rump} that implies that radical rings also form a subvariety $\mathsf{RadRng}$ of the variety $\mathsf{SKB}$ of skew braces.
Since 
skew braces form a variety of $\Omega$-groups in the sense of Higgins \cite{Higgins}, hence in particular a semi-abelian category \cite{JMT} (in fact, even a strongly protomodular category \cite{Bourn, MR4561474}), it is natural 
to look at \emph{non-abelian homology} of skew braces with coefficient functors in the subvarieties $\mathsf{Grp}$ of groups, $\mathsf{Ab}$ of abelian groups, and $\mathsf{RadRng}$ of radical rings.

Indeed, in recent years there have been some relevant developments in non-abelian homological algebra (see \cite{EVdL, EGVdL, Jan-Abstract, Protoadditive, Duckerts}, for instance, and the references therein). The comonadic approach to homology theory of algebraic structures \cite{BB} turns out to be perfectly compatible with the fundamental concept of semi-abelian category, allowing one to extend and improve some classical results in group homology to a general categorical context including compact groups, crossed modules, Lie algebras and cocommutative Hopf algebras, for instance.

This article is a first step in the direction of applying the above approach to non-abelian homology to the variety of skew braces, by providing some precise descriptions of the second homology group of a skew brace in terms of a generalized Hopf formula.

In order to briefly explain this, it is useful to first recall this classical formula in the category $\mathsf{Grp}$ of groups and its relationship with the notion of central extensions. Consider the adjunction 
\begin{equation}\label{groups}
\begin{tikzcd}
	{\textsf{Ab}} & \perp & {\textsf{Grp}}
	\arrow["U"', shift right=2, from=1-1, to=1-3]
	\arrow["{\textsf{ab}}"', shift right=3, from=1-3, to=1-1]
\end{tikzcd}
\end{equation}
where $\mathsf{ab} \colon \mathsf{Grp} \rightarrow \mathsf{Ab}$ is the classical abelianisation functor sending a group $G$ to the quotient $\mathsf{ab}(G) = \frac{G}{[G,G]}$ of $G$ by its derived subgroup $[G,G]$. Given a free presentation of a group $G$
\[
\begin{tikzcd}
	0 & K & F & G & 0
	\arrow[from=1-1, to=1-2]
	\arrow[tail, from=1-2, to=1-3]
	\arrow["f", two heads, from=1-3, to=1-4]
	\arrow[from=1-4, to=1-5]
\end{tikzcd} 
\]
where $F$ is a free group, consider the quotient $\frac{F}{[K,F]}$ of $F$ by the commutator subgroup $[K,F]$ as in the diagram
\begin{equation}
\label{Universal-ext}
\begin{tikzcd}
	F && F/[K,F] \\
	& G
	\arrow[two heads, from=1-1, to=1-3]
	\arrow["f"', two heads, from=1-1, to=2-2]
	\arrow["{\bar{f}}", two heads, from=1-3, to=2-2]
\end{tikzcd}
\end{equation}
where the induced morphism $\overline{f}$ is then a weakly universal central extension of $G$ \cite{BB}. The Galois group of this central extension  $\overline{f}$ of $G$ turns out to be an \emph{invariant} of $G$, called the \emph{fundamental group} $\pi_1(G)$ \cite{Jan-Abstract} of $G$, that is naturally isomorphic to the second integral homology group of $G$ 
$$\pi_1(G)\cong  \frac{K \cap [F,F]}{[K,F]} \cong \mathsf{H}_2 (G)$$
via the classical Hopf formula, i.e. the right-hand isomorphism. The commutators $[K,F]$ and $[F,F]$ appearing in this formula are ``relative'' to the chosen subvariety of the variety $\mathsf{Grp}$ of groups, in this case the variety $\mathsf{Ab}$ of abelian groups, so that they could also be denoted by $[K,F]_{\mathsf{Ab}}$ and $[F,F]_{\mathsf{Ab}}$, respectively. The homology group $\mathsf{H}_2 (G)$ can also be obtained by applying comonadic homology theory \cite{BB} to the variety $\mathsf{Grp}$ with respect to the abelianisation functor $\mathsf{ab} \colon \mathsf{Grp} \rightarrow \mathsf{Ab}$, so that $\mathsf{H}_2 (G) \cong \mathsf{H}_2 (G, \mathsf{ab})$, the right-hand side homology group now being the one arising from the comonadic approach by taking $\mathsf{ab}$ as coefficient functor:
\begin{equation}\label{Hopf-formula}
  \frac{K \cap [F,F]_{\mathsf{Ab}}}{[K,F]_\mathsf{Ab}} \cong \mathsf{H}_2 (G, \mathsf{ab})
  \end{equation}

The comonadic homology theory \cite{BB} in the semi-abelian context (as developed in \cite{EVdL, EGVdL}), allows one to establish several generalized Hopf formulas for homology in other semi-abelian varieties than the variety of groups, relatively to a given subvariety (not necessarily the one of ``abelian algebras'', as it is the case in the example above). The main difficulty in computing these formulas is to find an algebraic characterization of the central extensions (in the sense of \cite{janelidze1994galois}) with respect to the given subvariety, so that the denominator in the formula \eqref{Hopf-formula} can be made explicit. We do this in the present article in the variety $\mathsf{SKB}$ with respect to three particularly interesting subvarieties: the varieties $\mathsf{RadRng}$ of radical rings, $\mathsf{Grp}$ of groups and $\mathsf{Ab}$ of abelian groups. These are connected by the following adjunctions
\begin{equation}
\label{Adjunctions}
\begin{tikzcd}
	{\textsf{SKB}} & \perp & {\textsf{Grp}} \\
	\vdash && \vdash \\
	{\textsf{RadRng}} & \perp & {\textsf{Ab}}
	\arrow["F", shift left=2, from=1-1, to=1-3]
	\arrow["G", shift left=2, from=1-1, to=3-1]
	\arrow["{U_G}", shift left=2, from=1-3, to=1-1]
	\arrow["{\textsf{ab}}", shift left=2, from=1-3, to=3-3]
	\arrow["U_R", shift left=2, from=3-1, to=1-1]
	\arrow["{\hat{F}}", shift left=2, from=3-1, to=3-3]
	\arrow["U", shift left=2, from=3-3, to=1-3]
	\arrow["{U_A}", shift left=2, from=3-3, to=3-1]
\end{tikzcd}
\end{equation}
where $U_A, U_R$, $U_G$ and $U$ are inclusion functors, and $\hat{F}$, $G$, $F$ and $\mathsf{ab}$ their left adjoints.

 First, in Section \ref{Section-Grp}, we characterize the central extensions in $\mathsf{SKB}$ corresponding to the subvariety $\mathsf{Grp}$ of groups (Proposition \ref{Pro:CenterGroup}) that are the surjective homomorphisms $f \colon A \rightarrow B$ of skew braces such that $[\Ker(f),A]_{\mathsf{Grp}}=0$,
where $[\Ker(f),A]_{\mathsf{Grp}}$
is the additive subgroup of $A$ generated by all the elements of the form $$\lbrace a*b, \quad  b*a,\quad c+b*a-c\, \mid \, a\in \Ker (f ),b \in A ,c\in A \rbrace,$$ 
where $a*b=-a+a\circ b-b$, which is an ideal of $A$.

Section \ref{sect:Radrings} is devoted to the study of central extensions of skew braces relative to the subvariety $\mathsf{RadRng}$ of radical rings.
For this purpose we introduce a new object, the \emph{radicalator} (Definition \ref{def:radicalator}). 
 This is generated as a subgroup of the additive group by the elements $(a+b)\circ c -b\circ c +c -a\circ c$ and $a+b-a-b$ for all $a,b,c\in A$.
The radicalator turns out to be an ideal $[A,A]_{\mathsf{RadRng}}$ (Proposition \ref{RadIdeal}) playing a similar role with respect to the reflection from $\mathsf{SKB}$ to $\mathsf{RadRng}$ to the role played by the derived subgroup in the reflection from $\mathsf{Grp}$ to $\mathsf{Ab}$. Indeed, the quotient $A \rightarrow \frac{A}{[A,A]_{\mathsf{RadRng}}}$ is the universal one transforming a skew brace $A$ into a radical ring $\frac{A}{[A,A]_{\mathsf{RadRng}}}$, and an extension $f \colon A \rightarrow B$ of skew braces is central relatively to the subvariety $\mathsf{RadRng}$ if and only if $[\Ker(f), A]_{\mathsf{RadRng}} =0$ (Theorem \ref{Central-Radical}).
 This pushes forward Rump's connection between skew braces and radical rings 
\cite{Rump} (see also \cite{MR4136750}). We also observe that the central extensions of skew braces relative to the subvariety of braces (= skew braces of abelian type) can also be characterized in algebraic terms (Remark 
\ref{RemBraces}). 

In Section \ref{Section-Ab}, we consider the reflection from $\mathsf{SKB}$ to the subvariety $\mathsf{Ab}$ of abelian groups, and characterize the corresponding central extensions. For this, given an ideal $I$ of $A$, the additive subgroup generated by the set 
\[\lbrace [a,b]_+,\: a*b,\: [a,b]_\circ \quad \text{for all}\quad a\in I,  b\in A\rbrace.\]
is shown to be an ideal (Proposition \ref{Relative-Ab}), denoted by $[I,A]_{\mathsf{Ab}}$ (here $[a,b]_+$ and $[a,b]_\circ$ are the usual group-theoretic commutators with respect to the operations $+$ and $\circ$, respectively). An extension $f \colon A \rightarrow B$ of skew braces is then central with respect to $\mathsf{Ab}$ if and only if $[\Ker(f), A]_{\mathsf{Ab}}=0$ (Proposition \ref{Pro:CentralAbelian}).
Note that these central extensions of skew braces have already been studied in \cite{MR4688775,MR4604853,MR4697968}. We also observe that the commutator $[\Ker(f), A]_{\mathsf{Ab}}$ coincides with the Huq commutator defined in \cite{Huq} (Proposition \ref{Central-Huq}).

In the last section we deduce the corresponding Hopf formulae for homology of skew braces, and 
then apply the Stallings--Stammbach exact sequence holding in any semi-abelian variety \cite{EVdL, EGVdL} to these contexts. A result relating lower central series to low-dimensional homology concludes the article (Corollary \ref{series}).
This work opens the way to the study of \emph{relative commutators} in skew braces in the sense of \cite{MR2311168}, as well as to the theory of higher central extensions of skew braces, following the recent developments in categorical algebra \cite{EGVdL, Jan-Abstract, Protoadditive, Duckerts}.

\section{Preliminaries}
Recall that a \emph{skew brace} \cite{MR3647970} is a triple $(A,+,\circ)$, 
where $(A,+)$ and $(A,\circ)$ are groups such that
the compatibility condition 
$a\circ (b+c)=a\circ b-a+a\circ c$ holds for all $a,b,c\in A$. The inverse of an element $a\in A$ with respect to the circle 
operation $\circ$ will be denoted by $a'$. We will denote the groups $(A,+)$ and $(A,\circ)$ by $A_+$ and $A_\circ$, respectively.

Let $A$ be a skew brace. There are two canonical actions by automorphisms 
\begin{align*}
    &\lambda\colon A_\circ\to \Aut(A_+), &&\lambda_a(b)=-a+a\circ b,\\
    &\rho\colon A_\circ\to \Aut(A_+), &&\rho_a(b)=a\circ b-a.
\end{align*}

We can introduce a binary operation $* \colon A \times A \rightarrow A$ defined by $a*b=-a+a\circ b-b$ for all $a,b\in A$. 
By definition, a normal subgroup $I$ of $A_+$
is an \emph{ideal} of $A$ if 
$a*x \in I$ and $x*a \in I$ for all $a\in A$ and $x\in I$.
A normal subgroup $I$ of the group $(A,+)$ such that $I$ is a subgroup of $A_\circ$ and $A*I\subset I$ is called a \emph{strong left ideal}.

\begin{example}
    Let $A$ be a skew brace. Then the center $Z(A_+)$ is a strong left ideal of $A$. 
\end{example}


\begin{definition}
    We call \emph{extension} of $B$ a surjective morphism $f:A\to B$ in $\mathsf{SKB}$, and we also denote this extension by $(A,f)$.
\end{definition}

Since skew braces form a variety $\mathsf{SKB}$ of $\Omega$-groups in the sense of Higgins \cite{Higgins}, any subvariety $\mathcal{X}$ of $\mathsf{SKB}$ is admissible in the sense of Categorical Galois Theory \cite{janelidze1994galois}. We denote by $I \colon \mathsf{SKB}\to \mathcal{X}$ the reflector to the subvariety $\mathcal{X}$ of $\mathsf{SKB}$, left adjoint of the inclusion functor $U \colon \mathcal{X} \rightarrow \mathsf{SKB}$.
\begin{definition}
For any $A$ in $\mathsf{SKB}$, the kernel of the $A$-component $\eta_A \colon A \rightarrow UI(A)$ of the unit $\eta$ of the adjunction 
\begin{equation}\label{adjunction}
\begin{tikzcd}
	{\mathcal{X}} & \perp & {\mathsf{SKB}}
	\arrow["U"', shift right=3, from=1-1, to=1-3]
	\arrow["I"', shift right=3, from=1-3, to=1-1]
\end{tikzcd}
\end{equation}
will be denoted by $R(A)$, so that 
\[
\begin{tikzcd}
0 & R(A) & A & UI(A) & 0
\arrow[from=1-1, to=1-2]
\arrow[from=1-2, to=1-3]
\arrow[from=1-3, to=1-4]
\arrow[from=1-4, to=1-5]
\end{tikzcd}
\]
is an exact sequence in $\mathsf{SKB}$.
\end{definition}
Observe that $R \colon \mathsf{SKB} \rightarrow \mathsf{SKB}$ is an endofunctor of $\mathsf{SKB}$.

Given an extension $f:A\to B$ we form the pullback
\begin{equation}\label{central-ext}
\begin{tikzcd}
	{A\times_BA} & A \\
	A & B
	\arrow["t", from=1-1, to=1-2]
	\arrow["s"', from=1-1, to=2-1]
	\arrow["f", from=1-2, to=2-2]
	\arrow["f"', from=2-1, to=2-2]
\end{tikzcd}
\end{equation}
where $A \times_B A = \{ (a_1, a_2) \in A \times A \, \mid \, f(a_1)=f(a_2)\}$, while $s$ and $t$ are the first and second projection, respectively.
The definition of \emph{central extension} relative to a given subvariety $\mathcal X$ of $\mathsf{SKB}$ (as in the work of the Fr\"{o}hlich school \cite{Frolich, FC, Lue}) is then the following:
\begin{definition}
    For a subvariety $\mathcal{X}$ of the variety $\mathsf{SKB}$ of skew braces, an extension $(A,f)$ is \emph{central} if $R(s)=R(t)$.
\end{definition}
As observed in \cite{janelidze1994galois}, this definition coincides with the categorical one in the case of varieties of $\Omega$-groups, hence in particular in the variety $\mathsf{SKB}$.
Any subvariety $\mathcal X$ of $\mathsf{SKB}$ as above
is then \emph{admissible} in the sense of Categorical Galois Theory for the class of \emph{surjective homomorphisms} (see Theorem $3.4$ in \cite{janelidze1994galois}).
This means that the left adjoint $I$ in \eqref{adjunction} preserves a special type of pullbacks, namely the ones of the form
\[
\begin{tikzcd}
	{A\times_{UI(A)}U(B)} & U(B) \\
	A & UI(A)
	\arrow[from=1-1, to=1-2]
	\arrow[from=1-1, to=2-1]
	\arrow["\phi", from=1-2, to=2-2]
	\arrow["\eta_A"', from=2-1, to=2-2]
\end{tikzcd}
\]
where $\phi$ is any \emph{surjective} homomorphism and $\eta_A$ is the $A$-component of the unit $\eta$ of the adjunction. The property of admissibility guarantees the validity of a general Galois theorem \cite{Jan} that in our case allows one to classify the central extensions of a skew brace $B$ in terms of a suitable category of (internal) actions in $\mathcal X$ on the Galois groupoid of a weakly universal central extension $(E,p)$ of $B$ (see the last section in \cite{janelidze1994galois} for more details).

\section{The subvariety of groups}\label{Section-Grp}
By the classical Birkhoff theorem, the variety of groups can be identified with the subvariety $\mathsf{Grp}$ of $\mathsf{SKB}$ determined by the additional identity $x + y= x \circ y$. It then follows that there is an adjunction between these two categories
\begin{equation}\label{adjunction-groups}
\begin{tikzcd}
	{\mathsf{Grp}} & \perp & {\textsf{SKB}}
	\arrow["{U_G}"', shift right=3, from=1-1, to=1-3]
	\arrow["F"', shift right=3, from=1-3, to=1-1]
\end{tikzcd}
\end{equation}
where $\mathsf{Grp}$ is the subvariety of $\mathsf{SKB}$ whose objects are skew braces with the property that the two group operations $+$ and $\circ$ are equal. This subvariety is isomorphic to the category of groups, of course, and this justifies the slight abuse of notation. The functor $U_G$ is a full inclusion, and in order to describe its left adjoint $F$ in \eqref{adjunction-groups} the following definition will be needed:
\begin{definition}
    Let $A$ be a skew brace, we denote by $A*A$ the subgroup of $A_+$ generated by the elements $a*b$ for all $a,b\in A$. 
\end{definition}
It is well known that $A*A$ is an ideal of $A$ and it is the smallest ideal $I$ such that the quotient $A/I$ is a group.

In our situation the reflector $F: \mathsf{SKB}\to \mathsf{Grp}$ is the functor  sending the skew brace $A$ to the quotient $F(A) = A/(A*A)$ by the ideal $A*A$ (with obvious definition on morphisms). Note that here the ideal $A*A$ is precisely the kernel $R(A)$ of the $A$-component of the unit $\eta_A \colon A \rightarrow A/(A*A)$ of the adjunction \eqref{adjunction-groups}. This adjunction is admissible in the sense of Categorical Galois theory, since $\mathsf{Grp}$ is a subvariety of the variety $\mathsf{SKB}$ (see Section $5$ in \cite{janelidze1994galois}).

\begin{definition}
    Let $A$ be a skew brace and $I\subset A$ an ideal, we define $[I,A]_{\mathsf{Grp}}$ to be the additive subgroup of $A$ generated by the set
    \[\lbrace a*b, \quad  b*a,\quad c+b*a-c\, \mid \, a\in I,\; b,c \in A\rbrace.\]
\end{definition}

\begin{proposition}
    Let $A$ be a skew brace and $I\subset A$ an ideal of $A$. Then $[I,A]_{\mathsf{Grp}}$ is an ideal.
\end{proposition}

\begin{proof}
        Let $J=[I,A]_{\mathsf{Grp}}$.
        Since $I$ is an ideal, we have that $J\subset I$. Thus, by definition of $J$, it is clear that $J*A\subset J$ and $A*J\subset J$.
        Moreover, $$a*(c+b)= a*c+c+ a*b-c \in J$$ for all $a\in J$ and $b,c\in A$. Thus, $c+ a*b-c\in J$, so that $J$ is an ideal of $A$.
\end{proof}

\begin{remark}
    The additive subgroup of $A$ generated by the subset 
    \[\lbrace a*b, \quad  b*a\, \mid \, a\in I,b\in A\rbrace\] is not an ideal of $A$, in general. 
    The database of \cite{MR3647970} contains a (minimal) counterexample 
    of size 24.
\end{remark}

\begin{proposition}
\label{Pro:CenterGroup}
    An extension $(A,f)$ is central with respect to the adjunction \eqref{adjunction-groups} if and only if $[\Ker(f),A]_{\mathsf{Grp}}=0$.
\end{proposition}

\begin{proof}
    Assume that $(A,f)$ is a central extension. Take $A \times_B A$ as in \eqref{central-ext}. Clearly, for any $a\in \Ker(f)$, $(a,0)\in A \times_B A$. If $b\in A$, then $$(a,0)*(b,b) = (a*b,0)\in {R(A \times_ B A)}.$$ Since $(A,f)$ is central, by definition $a*b=R(s)(a*b,0) = R(t)(a*b,0) = 0$. Similarly, one can show that $b*a=0$. As $b \in A$ and $a \in \Ker(f)$ were chosen arbitrarily, it follows that $[\Ker(f),A]_{\mathsf{Grp}}=0$.

    Assume now that $[\Ker(f),A]_{\mathsf{Grp}}=0$. Let $(a,b)$ and $(a_1,b_1)$ be two elements in $A \times_B A$. Since $-a_1+ b_1 \in \Ker (f)$ one has that $a_1 \circ (-a_1+b_1) = b_1$ so that $-a_1+b_1=a_1'\circ b_1$. 
    It follows that
    \begin{align*}
        a*a_1-b*b_1 &=-a+a\circ a_1 -a_1+b_1-b\circ b_1 +b\\
        &=-a+a\circ a_1 +a_1'\circ b_1-b\circ b_1 +b \\
        &=-a+a\circ a_1 \circ a_1'\circ b_1-b\circ b_1 +b\\ 
        &= -a+ a\circ b_1-b\circ b_1 +b  \\
         &= -a+ (a\circ b'\circ b \circ b_1)  -b\circ b_1 +b\\
          &= -a+ (a\circ b'+ b \circ b_1)  -b\circ b_1 +b\\
        &= -a + a\circ b'+b\\
        &= -a +a \circ b' \circ b\\
        &=0,
    \end{align*}
    where we have used the fact that $a_1'\circ b_1 \in  \Ker (f)$ and $a\circ b' \in  \Ker (f)$. It follows that the restrictions $R(s)$ and $R(t)$ of the first and the second projections $s$ and $t$ (with the notations from diagram \eqref{central-ext}) are equal, as desired.
\end{proof}

\section{The subvariety of radical rings}
\label{sect:Radrings}
We write $\mathsf{RadRng}$ to denote the category of radical rings. As it follows from the results in \cite{Rump} this category can be presented as the subvariety of $\mathsf{SKB}$ determined by the following additional two identities:
$$ (a + b) \circ c = a \circ c - c + b \circ c$$
and 
$$ a + b = b + a.$$
This important observation implies that the variety $\mathsf{RadRng}$ of radical rings determines an adjunction
\begin{equation}\label{Radical}
\begin{tikzcd}
	{\mathsf{RadRng}} & \perp & {\textsf{SKB}}
	\arrow["{U_R}"', shift right=3, from=1-1, to=1-3]
	\arrow["G"', shift right=3, from=1-3, to=1-1]
\end{tikzcd}
\end{equation}
where $U_R$ is the inclusion functor and $G$ its left adjoint associating, with any skew brace $A$, its universal radical ring $G(A)$. This adjunction is again admissible in the sense of Categorical Galois theory \cite{janelidze1994galois} (essentially for the same reasons as for the subvariety $\mathsf{Grp}$ of $\mathsf{SKB}$, as explained in the previous section). 
The main goal of this section is to characterize the extensions in $\mathsf{SKB}$ that are \emph{central} with respect to the adjunction \eqref{Radical}.


\begin{definition}\label{radicalator}
    Let $A$ be a skew brace. We define the \emph{right distributor} of $a,b,c\in A$ to be the element 
    \[
    [a,b,c]=(a+b)\circ c -b\circ c +c -a\circ c.
    \]
\end{definition}
\begin{proposition}
    Let $A$ be a skew brace and $I$ an ideal of $A$. Then for all $c\in I$ and $a,b\in A$, the elements $[a,b,c]$, $[c,a,b]$ and $[b,c,a]$ are all in $I$.
\end{proposition}

\begin{proof}
Let $A\to A/I$, $a\mapsto \overline{a}$, be the canonical map. 
Then 
\[
\overline{[a,b,c]}=[\overline{a},\overline{b},\overline{c}]=[\overline{a},\overline{b},0]=0.
\]
Similarly, $\overline{[c,a,b]}=\overline{[b,c,a]}=0$.
\end{proof}

\begin{definition}
    Let $A$ be a skew brace, and $I\subset A$ an ideal. We define $[I,A]_{\mathsf{RadRng}}$ to be the additive subgroup of $A$ generated by the elements $[a,b,c]$, $[c,a,b]$, $[b,c,a]$ and $a+b-a-b$ for all $a\in I$ and $b,c\in A$.
\end{definition}

Our next goal is to show that $[I,A]_{\mathsf{RadRng}}$ is an ideal.

\begin{remark}
    If we write $J=[I,A]_{\mathsf{RadRng}}$, then $J_+$ is a normal subgroup of $A_+$. Indeed, for any $j \in J \subset I$ and $a \in A$, the element $j + a -j -a$ belongs to $J$, so that $J$ is stable under conjugation in $A$.
\end{remark}

\begin{lemma}\label{Radic-quotient}
    \label{lem:rightDistrib}
    Let $A$ be a skew brace, $I\subset A$ an ideal and $J=[I,A]_{\mathsf{RadRng}}$. Then $I_+/J_+ \subset Z(A_+/J_+)$.  In addition, if $a,b,c\in A$ and one of them is in $I$, then $$ \overline{(a+b)\circ c}=\overline{a\circ c-c+b\circ c} $$
    in $I_+/J_+.$
\end{lemma}
\begin{proof}
    This is a direct consequence of the definition of $J$.
\end{proof}
\begin{corollary}
\label{cor:rightDistrib}
   Let $A$ be a skew brace, $I\subset A$ an ideal and $J=[I,A]_{\mathsf{RadRng}}$. For all $a\in I$ and $b,d\in A$ we have the following identities in $A_+/J_+$:
   \begin{align*}
       \overline{(b\circ a-b)\circ d-d}&=\overline{b\circ a\circ d-b\circ d},\\
       \overline{(a\circ b-b)\circ d -d}&= \overline{a\circ b\circ d-b\circ d},\\
       \overline{(b-b\circ a)\circ d -d}&= \overline{b\circ d-b\circ a\circ d},\\
       \overline{(b-a\circ b)\circ d -d}&= \overline{b\circ d- a\circ b\circ d}.
   \end{align*}
\end{corollary}

\begin{proof}
    Apply Lemma \ref{lem:rightDistrib} to $\overline{(b\circ a-b+b)\circ d}$ and $\overline{(a\circ b-b+b)\circ d}$ and then to $\overline{(-(b\circ a-b))\circ d}$ and $\overline{(-(a\circ b-b))\circ d}$.
\end{proof}

\begin{lemma}
    Let $A$ be a skew brace, $I\subset A$ an ideal and $J=[I,A]_{\mathsf{RadRng}}$. Then $A*J\subset J$. In particular, the actions $\lambda$ and $\rho$ of $A$ factor through actions $\overline{\lambda}$ and $\overline{\rho}$ on $A_+/J_+$.
\end{lemma}

\begin{proof}
    To see that $\lambda_a(J)\subset J$, it suffices to check on the generators of $J$. It is clear that $\lambda_d([I_+,A_+])\subset [I_+,A_+]\subset J$ for all $d\in A$. Let then $d\in A$, we must show that $\overline{\lambda_d([a,b,c])}=0$ in the quotient group $A_+/J_+$ where $a,b,c\in A$ and at least one of these elements is in $I$.

First assume that $c\in I$.
    By Lemma \ref{lem:rightDistrib},
    \begin{align*}
    \label{eq:noname2}
        \overline{\lambda_d([a,b,c])} &= \overline{(d\circ a+\lambda_d(b))\circ c -d\circ b\circ c+d\circ c-d\circ a\circ c}\\
        &= \overline{(d\circ a+\lambda_d(b) -d\circ b+d-d\circ a)\circ c-c}\\
        &=0.
    \end{align*}
    Now, if we assume that $b\in I$, then  
    \[ 
    -c+\lambda_d(b)\circ c\in I,\quad 
    -d\circ b\circ c +d\circ c=\lambda_d(-b\circ c+c)\in I. 
    \]
    Thus 
    \begin{align*}
        \overline{\lambda_d([a,b,c])} &=\overline{d\circ a\circ c}-\overline{c}+ \overline{\lambda_d(b)\circ c} -\overline{d\circ b\circ c}+\overline{d\circ c}-\overline{d\circ a\circ c}\\
        &= -\overline{c}+\overline{\lambda_d(b)\circ c}-\overline{d\circ b\circ c}+\overline{d\circ c}\\
        &= -\overline{d\circ b\circ c}+\overline{d\circ c}-\overline{c}+\overline{\lambda_d(b)\circ c}\\
        &=-\overline{d\circ b\circ c}+\overline{(d+\lambda_d(b))\circ c}\\
        &=0.
    \end{align*}
    Similarly, if we assume that $a\in I$, 
    \begin{align*}
        \overline{\lambda_d([a,b,c])} &= \overline{(d\circ a -d)\circ c}-\overline{c}+\overline{d\circ c}-\overline{d\circ a\circ c}\\
        &= \overline{(d\circ a -d+d)\circ c}-\overline{d\circ a\circ c}\\
        &=0.\qedhere 
    \end{align*}
\end{proof}
\begin{proposition}\label{RadIdeal}
    Let $A$ be a skew brace, and $I\subset A$ an ideal. The subgroup $[I,A]_{\mathsf{RadRng}}$ of $A_+$ is an ideal of $A$.
\end{proposition}
\begin{proof}
    Let $J=[I,A]_{\mathsf{RadRng}}$. 
    It is left to show that $J*A\subset J$. It suffices to show that $k\circ b-b\in J$ for all $k\in J$ and $b\in A$. Since $J\subset I$, it is sufficient to show this on the generators of $J$. Let $a\in I$ and $b,c\in A$. By Lemma \ref{lem:rightDistrib}, 
    \[
        \overline{(b-a)\circ c}=\overline{(b-a-b+b)\circ c}=\overline{(b-a-b)\circ c}-\overline{c}+\overline{b\circ c}.
    \]
    Thus 
    \begin{align*}
    \overline{(b-a-b)\circ c-c}&=  \overline{(b-a)\circ c} - \overline{b\circ c } \\ 
    & = \overline{b \circ c - c + (-a) \circ c } - \overline{ b \circ c} \\ &= \overline{b\circ c}-\overline{a\circ c}+\overline{c}-\overline{b\circ c} \\
    &= - \overline{a\circ c} + \overline{c}.
     \end{align*}
    Hence 
    \begin{align*}
        \overline{[a,b]_+\circ c-c}&= \overline{a\circ c-c + (b-a-b)\circ c-c}=0.
    \end{align*}
    Let $d\in A$. By using Corollary \ref{cor:rightDistrib}, 
    \begin{align*}
        \overline{[a,b,c]\circ d-d}&=\overline{(b\circ((b'\circ a-b')\circ c -c)-b +c -a\circ c)\circ d}-\overline{d}\\
        &= \overline{(b\circ((b'\circ a-b')\circ c -c)-b)\circ d}-\overline{d} +\overline{(c -a\circ c)\circ d}-\overline{d}\\
        &= \overline{b\circ((b'\circ a-b')\circ c -c)\circ d}-\overline{b\circ d}+\overline{c\circ d -a\circ c\circ d}\\
        &= \overline{b\circ((b'\circ a-b')\circ c -c)\circ d-b}+\overline{b}-\overline{b\circ d}+\overline{c\circ d -a\circ c\circ d}.
    \end{align*}
    We have shown earlier that the action $\rho$ of $A$ factors through an action $\overline{\rho}$ on $A_+/J_+$. Thus  
   \begin{align*}
        \overline{[a,b,c]\circ d-d}&= \overline{\rho}_b\left(\overline{(b'\circ a-b')\circ c -c)\circ d}\right)+\overline{b}-\overline{b\circ d}+\overline{c\circ d -a\circ c\circ d}\\
        &=\overline{\rho}_b\left(\overline{(b'\circ a-b')\circ c\circ d -c\circ d+d}\right)+\overline{b}-\overline{b\circ d}+\overline{c\circ d -a\circ c\circ d}\\
        &=\overline{\rho}_b\left(\overline{b'\circ a\circ c\circ d-b'\circ c\circ d+d}\right)+\overline{b}-\overline{b\circ d}+\overline{c\circ d -a\circ c\circ d}\\
        &=\overline{a\circ c\circ d}-\overline{c\circ d}+\overline{b\circ d}-\overline{b\circ d}+\overline{c\circ d -a\circ c\circ d}\\
        &=0. 
    \end{align*}
    To obtain the second and third equality, we used Corollary \ref{cor:rightDistrib}. Assume now that $b\in I$ and $a,c,d\in A$. 
    Then
    \begin{align*}
        \overline{[a,b,c]\circ d-d}&=\overline{\left((a+b)\circ c-a\circ c -b\circ c +c +[a\circ c, -c+b\circ c]_+\right)\circ d}-\overline{d}\\
        &= \overline{\left((a+b)\circ c-a\circ c -b\circ c +c\right)\circ d}-\overline{d}\\
        &= \overline{\left(a\circ\left((-a'+a'\circ b)\circ c-c\right)-a-b\circ c+c\right)\circ d}-\overline{d}\\
        &=0.
    \end{align*}
    Finally, assume that $c\in I$ and $a,b,d\in A$. In this case,
    \begin{align*}
         \overline{[a,b,c]\circ d-d}&=\overline{\left(\left((a+b)\circ c-b-a\right)+a+b -b\circ c +c -a\circ c\right)\circ d}-\overline{d}\\
         &= \overline{\left(\left((a+b)\circ c-b-a\right)+b -b\circ c +c +a-a\circ c\right)\circ d}-\overline{d}\\
         &=\overline{[a,b,c\circ d-d]}\\
         &=0
    \end{align*}
    by Lemma \ref{lem:rightDistrib} and Corollary \ref{cor:rightDistrib}.
\end{proof}
\begin{definition}
\label{def:radicalator}
    The \emph{radicalator} of a skew brace $A$ is the ideal  
    $A_R=[A,A]_{\mathsf{RadRng}}$.
\end{definition}
    It is now clear that the reflector $F: \mathsf{SKB}\to \mathsf{RadRng}$ in \eqref{Radical} is the functor sending the skew brace $A$ to the quotient $F(A) = A/A_R$ by the radicalator $A_R$ of $A$ (with obvious definition on morphisms).

\begin{definition}
    For a skew brace $A$, let
    \[
    Z_R(A)=\lbrace z\in Z(A_+): [z,b,c]=[b,z,c]=[b,c,z]=0 \text{ for all } b,c\in A\rbrace.
    \]
\end{definition}


\begin{lemma}
\label{lem:Z_R}
    Let $A$ be a skew brace and $z\in Z_R(A)$.
    Then 
    \begin{align}
        \label{eq:formulas}[z+a,b,c] =[a,z+b,c]=[a,b,z+c]= [a,b,c]
    \end{align}
    for all $a,b,c\in A$. 
\end{lemma}

\begin{proof}
    Let $a,b,c\in A$. To prove 
    $[z+a,b,c]=[a,b,c]$, we use that $z\in Z(A_+)$ and compute
    \begin{align*}
        [z+a,b,c] &= (z+a+b)\circ c-b\circ c+c-(z+a)\circ c\\
        &=\left((a+b)+z\right)\circ c-b\circ c+c-(a+z)\circ c\\
        &=(a+b)\circ c-c+z\circ c-b\circ c+c-(a\circ c-c+z\circ c )\\
        &=[a,b,c]+a\circ c-c+b\circ c-c+z\circ c-b\circ c+c-z\circ c+c-a\circ c\\
        &=[a,b,c]+a\circ c-c+(b+z)\circ c-(z+b)\circ c+c-a\circ c\\
        &=[a,b,c].
    \end{align*}

    A similar calculation shows that 
    \[
    [a,z+b,c]=[a,b,c].
    \]

    Finally, for the remaining equality, we
    note that both 
    $a\circ z-a=\lambda_a(z)$ and  
    $b\circ z-b=\lambda_b(z)$ are central 
    elements in the additive group of $A$. Then, using the skew brace compatibility condition,  
    \begin{align*}
        [a,b,z+c] &= (a+b)\circ (z+c)-b\circ (z+c)+(z+c)-a\circ (z+c)\\
        &=(a\circ z-a)+(a+b)\circ c-b\circ c+c-a\circ c+(a-a\circ z)\\
        &=(a+b)\circ c-b\circ c+c-a\circ c\\
        &=[a,b,c].\qedhere 
    \end{align*}
\end{proof}


\begin{proposition}
    Let $A$ be a skew brace. Then $Z_R(A)$ is a subbrace of $A$. 
\end{proposition}

\begin{proof}
    It follows immediately from 
    Lemma \ref{lem:Z_R} that 
    $Z_R(A)$ is a subgroup of $Z(A_+)$.

    We now prove that 
    $Z_R(A)$ is a subgroup of $A_\circ$.
    Let $k,k_1\in Z_R(A)$ and $b,c\in A$. First,
    \begin{align*}
        k\circ k_1'+b &=(k-k_1+b\circ k_1)\circ  k_1'\\
        &= (b\circ k_1-k_1+k)\circ k_1'=b+k\circ k_1'.
    \end{align*}
    Hence, $k\circ k_1'\in Z(A_+)$. Then,
    \begin{align*}
    (k\circ k_1'+b)\circ c&=(k-k_1+ b\circ k_1)\circ (k_1'\circ c)\\
    &=k\circ k_1'\circ c -c+b\circ c.
    \end{align*}
    Therefore $[k'\circ k_1,b,c]=0$. Similarly one sees that $[b,k'\circ k_1,c]=0$.
    Finally,
    \[
    (b\circ k\circ k_1'-k\circ k_1'+c\circ k\circ k_1')\circ k_1=b\circ k-k+c\circ k=(b+c)\circ k.  
    \]
    Thus
    \[
    b\circ k\circ k_1'-k\circ k_1'+c\circ k\circ k_1' = (b+c)\circ k\circ k_1'.
    \]
    This implies that $[b,c,k\circ k_1']=0$, and
    the proof is concluded.
\end{proof}

\begin{theorem}\label{Central-Radical}
    An extension $(A,f)$ in $\mathsf{SKB}$ is central with respect to the adjunction \eqref{Radical} if and only if $\Ker(f)\subset Z_R(A)$ or equivalently if $[\Ker(f),A]_{\mathsf{RadRng}}=0$.
\end{theorem}
\begin{proof}
Assume that $f \colon A \rightarrow B$ is an extension in $\mathsf{SKB}$ that is central with respect to the adjunction \eqref{Radical}. We write $$A \times_B A = \{ (a_1,a_2) \in A \times A \, \mid \, f(a_1) = f(a_2)\}$$ for the ``object part'' of the pullback in diagram \eqref{central-ext}.
 Let us first prove that $a+b-a-b=0$ for any $a \in \Ker(f)$, $b\in A$. The elements $(a,0)$ and $(b,b)$  are in $A \times_B A$, hence
\begin{align*}
    (a,0) + (b,b) - (a,0) -(b,b) & = (a+b-a-b, b-b)\\
    &= (a+b-a-b,0)\in R(A \times_B A). 
    \end{align*}
By applying the restrictions $R(s)$ and $R(t)$ of the pullback projections we get
    $$a+b-a-b = R(s)(a+b-a-b,0)= R(t) (a+b-a-b,0) =0.$$
    Let then $a\in \Ker(f)$ and $b,c \in A$. We claim that
    $[a,b,c]=0$. The elements $(a,0)$, $(b,b)$ and $(c,c)$ are all in $A \times_B A$, so that 
    \begin{align*}
     ((a,0) + (b,b) ) \circ (c,c) &-(b,b) \circ (c,c) + (c,c) - (a,0) \circ (c,c)     \\
        & = (a+b, b) \circ (c,c) - (b\circ c, b\circ c) +(c,c) -(a \circ c, c) \\
        & = ((a+b)\circ c -b \circ c +c - a \circ c, b \circ c -b \circ c +c - c) \\
        & = ([a,b,c], 0) \, \in R( A \times_B A).
        \end{align*}
  The assumption that $R(s)=R(t)$ gives $$[a,b,c]= R(s)([a,b,c], 0) =  R(t)([a,b,c], 0)=0.$$
    By a similar argument one can check that $[c,a,b]=0$ and $[b,c,a]=0$
for any $a\in \Ker(f)$ and $b,c \in A$. This proves the first implication.
    
    Conversely, assume now that $[\Ker(f),A]_{\mathsf{RadRng}}=0$. 
    Let $(a,a+k)\in A\times_B A$, $(a_1,a_1+k_1)\in A\times_B A$ and $(a_2,a_2+k_2)\in A\times_B A$ be such that $k,k_1,k_2\in \Ker(f)$. 
   It will suffice to  prove that the equality $R(s) = R(t)$ holds for all the generators of $R(A \times_B A) = [A\times_B A,A \times_B A]_{\mathsf{RadRng}}$. For this, observe that the assumption $\Ker(f) \subset Z(A_+)$ directly implies that
   \begin{align*}
       [a+k, a_1+k_1] &= a+k + a_1 + k_1 - (a+k) - (a_1 + k_1) \\
       &= a + a_1 -a-a_1 \\
       &= [a, a_1].
   \end{align*}
   Then, as a consequence of Lemma \ref{lem:Z_R}, we obtain the equality 
   \[
   [a+k,a_1+k_1,a_2+k_2] = [a,a_1,a_2],
   \]
   which concludes the proof.
\end{proof}
\begin{remark}\label{RemBraces} The arguments used in the proof of Theorem \ref{Central-Radical} also provide a characterization of central extensions of skew left braces with respect to the subvariety $\mathsf{BR}$ of braces \cite{Rump}, as we now briefly explain. The subvariety $\mathsf{BR}$ of $\mathsf{SKB}$ is determined by the additional identity $a+b-a-b =0$, since braces are precisely left skew braces of abelian type \cite{MR3647970}. We then have the following adjunction 
\begin{equation}\label{Braces}
\begin{tikzcd}
	{\mathsf{BR}} & \perp & {\mathsf{SKB}}
	\arrow["{U_B}"', shift right=3, from=1-1, to=1-3]
	\arrow["\textsf{br}"', shift right=3, from=1-3, to=1-1]
\end{tikzcd}
\end{equation}
where $U_B$ is the inclusion functor and $\mathsf{br}$ its left adjoint sending a skew brace $A$ to the quotient $\mathsf{br}(A) = \frac{A}{ [A,A]_{\mathsf{BR}} }$, with $[ A,A]_{\mathsf{BR}}$ the \emph{ideal} of $A$ generated by all the commutators $[a,b] = a+b -a-b$ (for any $a,b \in A$). One can then see that an extension $f \colon A \rightarrow B$ of skew braces is central with respect to the subvariety $\mathsf{BR}$ of braces if and only if its kernel $\Ker(f)$ satisfies the condition $[\Ker(f), A]_{\mathsf{BR}}= 0$, where $[\Ker(f), A]_{\mathsf{BR}}$ is the ideal of $A$ generated by all the commutators of the form $[k,a]= k+a-k-a$, where $k \in \Ker(f), a \in A$. Indeed, by looking at the first part of the proofs of the two implications in Theorem \ref{Central-Radical} one realizes that the condition $[\Ker(f), A]_{\mathsf{BR}}= 0$ is indeed necessary and sufficient for an extension to be central with respect to the adjunction \eqref{Braces}.
Similar results can be obtained by considering the subvarieties $\mathsf{Nil_nSKB}$ and $\mathsf{Sol_nSKB}$ of $\mathsf{SKB}$, where $\mathsf{Nil_nSKB}$ denotes the variety of skew left braces of $n$-nilpotent type and $\mathsf{Sol_kSKB}$ the variety of skew braces of $n$-solvable type ($n \ge 1$ is a positive integer). In the group case the characterizations of all these types of relative central extensions were established in \cite{EGVdL} (Section $9$). We shall not pursue this further in the present article, but it would be interesting to describe the corresponding relative commutators of skew braces on the model of what was done in \cite{RelComm} in the category of groups.
\end{remark}

\section{The subvariety of abelian groups}
\label{Section-Ab}

\begin{definition}
    Let $A$ be a skew brace, we denote by $A'$ the subgroup of $A_+$ generated by the elements $a*b=-a+a\circ b -b$ and $a+b-a-b$ for all $a,b\in A$. 
\end{definition}

A direct calculation shows that $A'$ is an ideal of $A$ and it is the smallest ideal $I$ such that the quotient $A/I$ is an abelian group; see \cite{MR4697968}.
The reflector $F: \mathsf{SKB}\to \mathsf{Ab}$ in the adjunction 
\begin{equation}\label{adjunction-Ab}
\begin{tikzcd}
	{\mathsf{Ab}} & \perp & {\textsf{SKB}}
	\arrow["{U}"', shift right=3, from=1-1, to=1-3]
	\arrow["F"', shift right=3, from=1-3, to=1-1]
\end{tikzcd}
\end{equation}
sends the skew brace $A$ to the quotient $F(A) = A/A'$ of $A$ by the ideal $A'$ (with obvious definition on morphisms). From the categorical point of view, the reflector $F$ in this adjunction gives the ``abelianisation functor'', since the (internal) abelian objects in the variety of skew braces are precisely the skew braces for which the two group structures coincide and are abelian (\cite{MR4561474}). Note that in the literature these skew braces are also referred to as \emph{trivial} braces.

\begin{definition}\label{RelativeCentr}
    Let $A$ be a skew brace, and let $I\subset A$ be an ideal of $A$. We define the relative commutator $[I,A]_{\mathsf{Ab}}$ to be the additive subgroup generated by the set
    \[\lbrace [i,b]_+,\: i*b,\: [i,b]_\circ\mid  i\in I,  b\in A\rbrace.\]
\end{definition}
\begin{proposition}\label{Relative-Ab}
    Let $A$ be a skew brace, and $I\subset A$ an ideal of $A$. Then, $[I,A]_{\mathsf{Ab}}$ is an ideal of $A$. 
\end{proposition}
\begin{proof}
   Since $J=[I,A]_{\mathsf{Ab}}\subset I$, $J$ contains the commutator subgroup $[J_+,A_+]$ so $J_+$ is normal in $A_+$. In addition, $J*A\subset J$. To see that $A*J\subset J$, it is enough to show that for $k\in J$ and $b\in A$, $b*k\in J$. Since $-b+b\circ k\in I$, in $A_+/J_+$ one has 
   \[
   \overline{-b+b\circ k}=\overline{b} +\overline{-b+b\circ k}-\overline{b}=\overline{b\circ k-b}.
   \]
   Therefore,
    \[\overline{b*k}=-\overline{k}+\overline{b\circ k}-\overline{b}=-\overline{k}+\overline{k\circ b}-\overline{b}  =  \overline{ k*b} = 0.\qedhere
    \]
\end{proof}
\begin{proposition}
\label{Pro:CentralAbelian}
    An extension $(A,f)$ is central with respect to the adjunction \eqref{adjunction-Ab} if and only if $[\Ker(f),A]_{\mathsf{Ab}}=0$.
\end{proposition}
\begin{proof}
    Assume first that $(A,f)$ is a central extension. Take $A \times_B A$ as in \eqref{central-ext}. Let $a\in \Ker(f)$ and $b\in A$. Then $(a,0)\in A \times_B A$ and
    \[
    (a,0)*(b,b) = (a*b,0)\in {R(A \times_ B A)}.
    \]
    Since $(A,f)$ is central, by definition $a*b=R(s)(a*b,0) = R(t)(a*b,0) = 0$. The first part of the proof of Theorem \ref{Central-Radical}, where it is shown that $[a,b]_+=0$, still applies here. The proof that $[a,b]_\circ=0$ is also similar, and is left to the reader.

    Conversely, let $(a,b)$ and $(a_1,b_1)$ be two elements in $A \times_B A$. As we saw in Proposition \ref{Pro:CenterGroup}, one has $a*a_1=b*b_1$. Then,
    \begin{align*}
        [a,a_1]_+-[b,b_1]_+ &=a+a_1-a-a_1+b_1+b-b_1-b\\
        &=a+a_1-a+b-a_1-b\\
        &=a+a_1-a_1-a+b-b\\
        &=0.
    \end{align*}
    Similarly, one has that $[a,a_1]_\circ\circ [b,b_1]_\circ' =0$, which implies that the restrictions $R(s)$ and $R(t)$ of the first and the second projections $s$ and $t$ (with the notations from diagram \eqref{central-ext}) are equal, as desired.
\end{proof}

One can ask whether the commutator $[I,A]_{\mathsf{Ab}}$ defined above coincides with the Huq commutator of normal subobjects
defined in \cite{Huq}. We shall now see that this is indeed the case.

Given an ideal $I$ of a skew brace $A$,
consider the set-theoretic map $c \colon I \times A \rightarrow A$ defined by $c(i,a) = i+a$ for any $i\in I$ and $a \in A$.
 The Huq commutator $[I,A]_{\mathsf{Huq}}$  is the smallest ideal $J$ of $A$
such that the composite map \begin{equation}\label{quotient}  
\begin{tikzcd}
	{I\times A} & A & {A/J}
	\arrow["c", from=1-1, to=1-2]
	\arrow["\pi", from=1-2, to=1-3]
\end{tikzcd}
\end{equation}  
is a \emph{homomorphism} of skew braces (here $\pi$ is the canonical quotient defined by $\pi(a)=\overline{a}$).
\begin{proposition}\label{Central-Huq}
 For any ideal $I$ of a skew brace $A$, 
 $$[I,A ]_{\mathsf{Huq}}= [I,A ]_{\mathsf{Ab}}.$$
\end{proposition}
\begin{proof}
    By definition of $[I,A ]_{\mathsf{Ab}}$, the set-theoretic map 
\[\begin{tikzcd}[column sep=scriptsize]
	{I\times A} & A & {A/[I,A]_{\textsf{ab}}}
	\arrow[from=1-1, to=1-2]
	\arrow[from=1-2, to=1-3]
\end{tikzcd}\]
    defined by $$\phi(i,a) = \overline{i} + \overline{a}$$ is a  homomorphism of skew braces, so that $[I,A ]_{\mathsf{Huq}} \subset [I,A ]_{\mathsf{Ab}}.$ For the  other inclusion, consider any quotient $\pi \colon A \rightarrow A/J$ having the property that the composite map $\pi c$ in \eqref{quotient} is a homomorphism of skew braces. For any $i \in I, a\in A$,
   $$\overline{i} \circ \overline{a} = \pi  c(i,0) \circ \pi c (0,a) = \pi c(i,a) = \pi c (0,a) \circ \pi c (i,0) = \overline{a} \circ \overline{i}, $$
   hence $\overline{[i,a]_{\circ}} = \overline{0}$.
   Similarly, one checks that $\overline{[i,a]_{+}} = \overline{0}$, and  $\overline{i*a}=\overline{0}$, so that
   all the generators appearing in the Definition \ref{RelativeCentr} of $[I,A ]_{\mathsf{Ab}}$ must be sent to $\overline{0}$ by $\pi$. It follows that $[I,A ]_{\mathsf{Ab}} \subset J$, hence in particular $[I,A ]_{\mathsf{Ab}} \subset [I,A ]_{\mathsf{Huq}}$, as desired.
\end{proof}
\begin{remark}
A different and equivalent description of the Huq commutator of two ideals in the category $\mathsf{SKB}$ was given in  \cite{MR4561474}, where the Huq commutator was also shown to coincide with the Smith commutator. 
\end{remark}

\section{Hopf formulae for homology}

The characterization of the central extensions obtained in the previous sections will now provide some new Hopf formulae for the homology of skew braces. Indeed, the variety $\mathsf{SKB}$ of skew braces is a variety of $\Omega$-groups and the subcategories $\mathsf{Grp}$, $\mathsf{RadRng}$, $\mathsf{BR}$ and $\mathsf{Ab}$ are all subvarieties so that the methods of \cite{Frolich, Lue, FC} apply, as also the recent and more general ones developed in the semi-abelian context \cite{EVdL, EGVdL, EG}.

The way one defines the (comonadic) homology of an algebra $B$ in a semi-abelian variety $\mathbb C$ relatively to a subvariety $\mathbb X$ of $\mathbb C$ 
\begin{equation}
\label{GeneralAdj}
\begin{tikzcd}[column sep=scriptsize]
	{\mathbb{X}} & \perp & {\mathbb{C}}
	\arrow["U"', shift right=3, from=1-1, to=1-3]
	\arrow["F"', shift right=3, from=1-3, to=1-1]
\end{tikzcd}
\end{equation}
can be briefly explained as follows (we refer the reader to \cite{EVdL, EGVdL} for more details).
The forgetful functor sending the algebra $B$ to its underlying set $\mid B \mid$ has a left adjoint, the ``free algebra functor'', naturally inducing a comonad 
\[
({\mathbb G} \colon \mathbb C \rightarrow \mathbb C, \epsilon \colon  G \Rightarrow 1_{\mathbb C}, \delta \colon  G \rightarrow G^2)
\]
on $\mathbb C$. The axioms of a comonad allows one to build a simplicial object ${\mathbb G}(B)$ in $\mathbb C$ \cite{BB}, with the standard ``free presentation'' of $B$ being given by $\epsilon_B \colon G(B) \rightarrow B$.
The ``homology algebras'' $\mathsf{H}_i(B,F)$ of $B$ (with coefficients in the reflector $F \colon \mathbb C \rightarrow \mathbb X$ to the subvariety $\mathbb X$) are the ones of the chain complex $N(F({\mathbb G}(B))$ obtained from the simplicial object $F({\mathbb G}(B))$ in $\mathbb X$ via the ``Moore normalization'' functor associating a chain complex  with any simplicial object in $\mathbb X$.

In the special case of the reflector $F \colon \mathsf{SKB} \rightarrow \mathsf{Grp}$, 
consider a ``free'' presentation 
\[
\begin{tikzcd}
	0 & K & F &B & 0
	\arrow[from=1-1, to=1-2]
	\arrow[tail, from=1-2, to=1-3]
	\arrow["f", two heads, from=1-3, to=1-4]
	\arrow[from=1-4, to=1-5]
\end{tikzcd} 
\]
of a skew brace $B$, where $F = {\mathbb G}(B)$ is the ``free'' skew brace on the underlying set of $B$. 
Consider also the ideal $[F,F]_{\mathsf{Grp}} = F*F$ of $F$ as defined in Section \ref{Section-Grp} and the ideal $$[K,F]_{\mathsf{Grp}} = \langle \lbrace a*b, \quad  b*a,\quad c+b*a-c\, \mid \, a\in K,b \in A, c\in A \rbrace  \rangle_{F}.$$
Then the first homology group $\mathsf{H}_1(B,F)$ of the skew brace $B$ (with coefficients in $F$) is given by $$\mathsf{H}_1 (B,F) = \frac{B}{[B,B]_{\mathsf{Grp}}},$$
while the
second homology group of $B$ is given by the Hopf formula 
$$\mathsf{H}_2 (B,F) \cong \frac{K \cap [F,F]_{\mathsf{Grp}}}{[K,F]_{\mathsf{Grp}}}.$$
In particular, this latter is an \emph{invariant} of the skew brace $B$, since it does not depend on the choice of the free presentation. 
According to the results in the above mentioned articles on the so-called \emph{$5$-term exact sequence}, also referred to as the  \emph{Stallings--Stammbach sequence}, we then get the following:
\begin{theorem}\label{SS}
Any short exact sequence 
\[
\begin{tikzcd}
	0 & K & A & B & 0
	\arrow[from=1-1, to=1-2]
	\arrow[tail, from=1-2, to=1-3]
	\arrow["f", two heads, from=1-3, to=1-4]
	\arrow[from=1-4, to=1-5]
\end{tikzcd} 
\]
in the variety $\mathsf{SKB}$ of skew braces induces the following exact sequence in $\mathsf{Grp}$:
\[
\begin{tikzcd}[column sep=scriptsize]
	{\mathsf{H}_2(A,F)} & {\mathsf{H}_2(B,F)} & {\frac{K}{[K,A]_{\textsf{Grp}}}} & {\mathsf{H}_1(A,F)} & {\mathsf{H}_1(B,F)} & 0.
	\arrow["\mathsf{H}_2(f)", two heads, from=1-1, to=1-2]
	\arrow[from=1-2, to=1-3]
	\arrow[from=1-3, to=1-4]
	\arrow["\mathsf{H}_1(f)", two heads, from=1-4, to=1-5]
	\arrow[from=1-5, to=1-6]
\end{tikzcd}
\]
\end{theorem}
The same kind of results can be stated by replacing $[F,F]_{\mathsf{Grp}} = F*F$ with $[F,F]_{\mathsf{RadRng}}$ and $[K,F]_{\mathsf{Grp}}$ with $[K,F]_{\mathsf{RadRng}}$, so that the second homology radical ring 
${\mathsf H}_2(B, G)$ of $B$ (with coefficients in the reflector $G \colon \mathsf{SKB} \rightarrow \mathsf{RadRng}$) is given~by
$$\mathsf{H}_2 (B,G) \cong \frac{K \cap [F,F]_{\mathsf{RadRng}}}{[K,F]_{\mathsf{RadRng}}}.$$

In case one chooses the subvariety $\mathsf{BR}$ of braces, one gets a new Hopf formula for the second homology of a skew brace
$$\mathsf{H}_2 (B,\mathsf{br}) \cong \frac{K \cap [F,F]_{\mathsf{BR}}}{[K,F]_{\mathsf{BR}}},$$
with coefficient functor $\mathsf{br} \colon \mathsf{SKB} \rightarrow \mathsf{BR}$, where the relative commutators $[F,F]_{\mathsf{BR}}$ and $[K,F]_{\mathsf{BR}}$ are described in Remark \ref{RemBraces}.


From this, as shown in \cite{EG}, one can also deduce a result concerning the lower central series. We explain here only the case of the subvariety $\mathsf{Ab}$ of abelian groups, however the same method can be applied to the subvarieties $\mathsf{Grp}$ of groups and $\mathsf{RadRng}$ of radical rings. One defines $A^0 =A$, and the ($n+1$)th term of the series inductively by $A^{n+1} = [A^n, A]_{\mathsf{Ab}}$, for any $n \ge 1$. Note that each $A^{n+1} = [A^n, A]_{\mathsf{Ab}}$ is an ideal of $A$ by Proposition \ref{Relative-Ab}. Note that these are central series 

\begin{corollary}\label{series}
Given a morphism $f \colon A \rightarrow B$ in $\mathsf{SKB}$, assume that $$\mathsf{H}_1(f) \colon \mathsf{H}_1(A,\mathsf{ab} ) \rightarrow \mathsf{H}_1(B,\mathsf{ab})$$ is an isomorphism and $$\mathsf{H}_2(f) \colon \mathsf{H}_2(A,\mathsf{ab} ) \rightarrow \mathsf{H}_2(B,\mathsf{ab})$$ is surjective. Then, for any $n\ge 1$, the induced morphism $\frac{A}{A^n} \rightarrow \frac{B}{B^n}$ is an isomorphism.
\end{corollary}

\subsection*{Acknowledgements}

This work was partially supported by
the project OZR3762 of Vrije Universiteit Brussel, the 
FWO Senior Research Project G004124N,
and the Fonds de la Recherche Scientifique  under Grant CDR No.
J.0080.23. Letourmy is supported by FNRS. 

\bibliographystyle{abbrv}
\bibliography{refs}

\end{document}